\documentclass[10pt,a4paper]{amsart}
\usepackage[utf8]{inputenc}
\usepackage{palatino, amsmath, amssymb}
\usepackage[colorlinks=false, breaklinks=true]{hyperref}
\newtheorem{theorem}{$\quad$Theorem}[section]
\newtheorem{lemma}[theorem]{$\quad$Lemma}

\title[About categorification of cyclotomic integers and tensored $N$-complexes]{About categorification of cyclotomic integers and \\ tensored \texorpdfstring{$N$}{N}-complexes}
\author{Djalal Mirmohades}

\begin{document}

\begin{abstract}
 We prove that the ideal used in recent works to categorify the cyclotomic integers is generated by a cyclotomic polynomial.
 Moreover, we publish a proof by T. Ekedahl that the $q$-binomial relations used in the tensor product of $N$-com\-plexes makes it necessary for the category to be enriched over the cyclotomic integers.
\end{abstract}

\maketitle

\section{Acknowledgements}

It was during my master's thesis that I started to investigate the question posed as Theorem \ref{th_principal}.
I could not prove it, nor find any counterexamples.
Finally, the problem reached Ekedahl in October 2010 who swiftly replied by mail with the proof published here.

About four years later at a conference in Montreal, Khovanov asked if someone could work on categorification of the ring of cyclotomic integers.

Motivated by Ekedahls proof, I started to look for a suitable ideal in a some product category of $N$-complexes so that I could apply the methods in the proof by Ekedahl to prove that the decategorified ideal is generated by the cyclotomic polynomial. The resulting paper \cite{mir} however only dealt with two distinct primes, and a simple calculation was sufficient to show that the ideal is principal.

However, without Ekedahls solution, I would most likely not have worked on \cite{mir}, and I would certainly not have written this paper. So I would like to acknowledge my gratitude to Torsten Ekedahl.

\section{About categorification of cyclotomic integers}

Recently, Laugwitz \& Qi in \cite{lq} constructed a monoidal category toghether with a thick ideal such that in the ring ${\mathbb Z}[q, q^{-1}]$, this ideal is generated by certain elements of the form $(q^m - 1)/(q^{m/p_k} - 1), \: p_k{|}m$.
Here, we show that this ideal is generated by $\Phi_m(q)$.
This result generalizes the case in \cite[Last step of main Theorem]{mir} where $m$ is equal to a product of two primes.
Note that we can do the calculations in ${\mathbb Z}[q]$ instead of ${\mathbb Z}[q, q^{-1}]$ because in the quotient ring we have $q^m = 1$, hence $q^{-1} = q^{m-1}$.

We refer to \cite{lq, mir} and the references within them for a more detailed introduction to the subject.

\begin{lemma}\label{lemma_maxprime}
Any maximal ideal of ${\mathbb Z}[x]$ contains a prime number.
\end{lemma}

\begin{proof}[Proof (unknown origin)] 
Let $I$ be an ideal of ${\mathbb Z}[x]$. If $I \cap {\mathbb Z} = (0)$, we need to show that $I$ cannot be maximal. Let $I'$ be the ideal of ${\mathbb Q}[x]$ generated by $I$. We have $I' = (f(x))$ for a polynomial $f$ coming from ${\mathbb Z}[x]$ of content $1$. The polynomial $f$ has degree $> 0$, because if $1$ could be written as a linear combination of polynomials from $I$ in ${\mathbb Q}[x]$, then $I \cap {\mathbb Z} \neq (0)$. But any $h \in I$ is of the form $gf$ for some $g \in {\mathbb Z}[x]$. It follows from Gauss lemma that $g \in {\mathbb Z}[x]$. This shows that $I \subseteq (f(x))$ in ${\mathbb Z}[x]$. To see that $(f(x))$ is not maximal in ${\mathbb Z}[x]$, pick an $m \in {\mathbb Z}$ such that $f(m) \neq 0, \pm 1$. Evaluation at $m$ then induces a map $${\mathbb Z}[x] \longrightarrow {\mathbb Z}/(f(m))$$ which has kernel strictly between $(f(x))$ and ${\mathbb Z}[x]$. Hence, any maximal ideal $I$ of ${\mathbb Z}[x]$ has $I \cap {\mathbb Z} = (n)$ for some integer $n > 0$. But since $n$ is the characteristic of the field ${\mathbb Z}[x]/I$, it must be prime.
\end{proof}

\begin{lemma}\label{lemma_ekedahl_method}
An ideal $I$ of ${\mathbb Z}[x]$ is the unit ideal if and only if, for every prime $p$, the canonical homomorphism ${\mathbb Z}[x] \to {\mathbb Z}/(p)[x]$ maps $I$ to the unit ideal in ${\mathbb Z}/(p)[x]$.
\end{lemma}

\begin{proof}
The canonical homomorphisms ${\mathbb Z}[x] \to {\mathbb Z}/(p)[x]$ maps the unit ideal to a unit ideal.
So assume $I$ is not the unit ideal of ${\mathbb Z}[x]$, then by Zorn's lemma $I$ lies in a maximal ideal $M$, and by Lemma \ref{lemma_maxprime}, there is a prime $p \in M$. But then, the image of $M$ in ${\mathbb Z}/(p)[x]$ cannot be the unit ideal.
\end{proof}

\begin{lemma}\label{lemma_gcd}
For a prime $p$, and distinct positive integers $n$ and $m$ not divisible by $p$, we have 
$\mathrm{gcd}\!\left\{\Phi_n (q),  \: \Phi_m (q) \right\} = 1$
in the principal ideal domain ${\mathbb Z}/(p)[q]$.
\end{lemma}

\begin{proof}
Let $\varphi$ denote Euler's totient function and $k$ denote the least common multiple of $n$ and $m$.
The integer $p^{\varphi(k)} - 1$ is then divisible by $k$, $n$ and $m$.
Let $\mathrm{GF}_{p^{\varphi(k)}}$ denote a splitting field of $q^{p^{\varphi(k)}} - q$ over ${\mathbb Z}/(p)$. 
Since the polynomial $\Phi_n (q) \Phi_m (q)$ is a divisor of $(q^{p^{\varphi(k)}-1} - 1)q$ in ${\mathbb Z}[q]$, it has only simple roots in $\mathrm{GF}_{p^{\varphi(k)}}$. This proves the lemma.
\end{proof}

\begin{theorem}
Let $n = p_1 p_2 \cdots p_t$, where $p_k$ are distinct primes.
In the ring ${\mathbb Z}[q]$, we have the following equality of ideals
$$
\left(\frac{[n]_q}{[n/p_1]_q}\right) + 
\left(\frac{[n]_q}{[n/p_2]_q}\right) + 
\cdots + 
\left(\frac{[n]_q}{[n/p_t]_q}\right) 
= \big( \Phi_n(q) \big)
$$
where $[m]_q = (q^m-1)/(q-1)$ and $\Phi_n$ denotes the $n$:th cyclotomic polynomial.
\end{theorem}

\begin{proof}
Recall that $q^n - 1 = \prod_{d{|}n} \Phi_d(q)$, hence 
$$[n]_q\big/[n/p_k]_q = (q^n-1)\big/(q^{n/p_k}-1) = \prod_{\substack{d{|}n,\\p_k{|}d}} \Phi_d(q).$$
Since $\Phi_n(q)$ divides every generator $[n]_q\big/[n/p_k]_q$, it is equivalent to show that the polynomials $[n]_q\big/[n/p_k]_q\Phi_n(q), 1 \leq k \leq t$ generate the unit ideal in ${\mathbb Z}[q]$.
We use Lemma \ref{lemma_ekedahl_method}, so we need to show this in ${\mathbb Z}/(p)[q]$ for an arbitrary prime $p$. 
Since ${\mathbb Z}/(p)[q]$ is a PID, it is enough to show that $\mathrm{gcd}([n]_q\big/[n/p_1]_q, \cdots, [n]_q\big/[n/p_t]_q) = \Phi_n(q)$.

In the case $p {\not|\:} n$, the minimum multiplicity of $\Phi_m(q)$ in $[n]_q\big/[n/p_k]_q$ over $1 \leq k \leq t$ is equal to $1$ if $m = n$ and equal to $0$ otherwise.
It then follows from Lemma \ref{lemma_gcd} that the $\mathrm{gcd}$ of the generators is equal to $\Phi_n(q)$.

In the case $p {\:|\:} n$, we again use Lemma \ref{lemma_gcd}.
We only need to count the multiplicities of $\Phi_m(q)$ where $p {\not|\:} m$, because by \cite[p. 160]{nag} (assuming $p {\not|\:} m$) and the fact that we work over characteristic $p$, we have
$$\Phi_{pm}(q) = \dfrac{\Phi_{m}(q^p)}{\Phi_{m}(q)} = \Phi_{m}(q)^{p-1}.$$
Moreover, when $p_k = p$
$$[n]_q\big/[n/p_k]_q = (q^n-1)\big/(q^{n/p}-1) = (q^{n/p}-1)^{p-1} = \prod_{d{|}n/p} \Phi_d(q)^{p-1}$$
and when $p_k \neq p$
$$[n]_q\big/[n/p_k]_q = (q^{n/p}-1)^p\big/(q^{n/p_kp}-1)^p = 
%\left([n/p]_q\big/[n/p_kp]_q\right)^p = 
\prod_{\substack{d{|}n/p,\\p_k{|}d}} \Phi_d(q)^p.$$

The multiplicity of $\Phi_{n/p}(q)$ in $[n]_q\big/[n/p_k]_q$ is then equal to $p-1$ when $p_k = p$ and equal to $p$ otherwise.
Hence the minimum multiplicity over $1 \leq k \leq t$ is equal to $p-1$, which is equal to the multiplicity of $\Phi_{n/p}(q)$ in $\Phi_{n}(q)$.
The minimum multiplicity of $\Phi_{m}(q)$, where $m \neq n/p$ and $p {\not|\:} m$ is equal to $0$ for the same reason as before.
\end{proof}

\section{About tensored \texorpdfstring{$N$}{N}-complexes}

To define a tensor product for $N$-complexes, Kapranov \cite{kapr91} uses $q$-com\-muta\-ti\-vi\-ty in the construction of the total complex.
Assuming that $q$ is a primitive $N$:th root of unity (that is $\Phi_N(q) = 0$) implies that $\binom{N}{i}_q = 0$ for $0 < i < N$.
This turns the total complex into an $N$-complex. See \cite[Prop. 1.8--1.10]{kapr91} for details.

The following theorem shows that the assumption that $\Phi_N(q) = 0$ is not only sufficient in the above construction, but also necessary.

\begin{theorem}\label{th_principal}
In the ring ${\mathbb Z}[q]$, we have the following equality of ideals
$$
\left(\left(\kern-1ex\begin{array}{c}n\\1\end{array}\kern-1ex\right)_q\right) + 
\left(\left(\kern-1ex\begin{array}{c}n\\2\end{array}\kern-1ex\right)_q\right) + 
\cdots + 
\left(\left(\kern-1ex\begin{array}{c}n\\n-1\end{array}\kern-1ex\right)_q\right) 
= \big( \Phi_n(q) \big)
$$
where $\binom{n}{i}_q$ denote $q$-binomial coefficients and $\Phi_n$ denotes the $n$:th cyclotomic polynomial.
\end{theorem}

\begin{proof}[Proof (T. Ekedahl)] 
The claim is that the ideal in ${\mathbb Z}[q]$ generated by $\binom{n}{i}_q, 0 < i < n$, is equal to $\Phi_n(q)$, the n'th
cyclotomic polynomial (all $\binom{n}{i}_q$ are clearly divisible by $\Phi_n(q)$ as the factor $\Phi_n(q)$ in $q^n-1$ appearing in the numerator doesn't cancel from the denominator). Hence an equivalent formulation
is that the ideal $I$ generated by $\binom{n}{i}_q/\Phi_n(q)$ is equal to the unit ideal. If not it is contained 
in a maximal ideal and any maximal ideal of ${\mathbb Z}[q]$ contains a prime number $p$. Hence we may
replace ${\mathbb Z}[q]$ by ${\mathbb Z}/p[q]$. As the latter ring is a PID, what we need to show is that the GCD of
the $\binom{n}{i}_q$ is equal to $\Phi_n(q)$. Now we recall that $q^m-1 = \prod_{d|m} \Phi_d(q)$ and if $m = p^km'$ we have
$\Phi_m = \Phi_{m'}(q)^{p^k}$ (everything computed in ${\mathbb Z}/p[q]$). Finally, if $p {\not|\:} m, m'$, then $\Phi_m(q)$ and $\Phi_{m'}(q)$ 
are relatively prime. Hence, for $p {\not|\:} d$ we have that the multiplicity with which $\Phi_d$ divides $q^m-1$
is equal to $1$ if $d {\not|\:} m$ and equal to $\psi(k) := p^k+p^{k-1}+\cdots+1$ if $d | m$ and $k$ is the largest power
of $p$ dividing $m$.

To show the result it is enough to show that for every $d$ with $p {\not|\:} d$ the largest power of $\Phi_d(q)$
which divides all $\binom{n}{i}_q$ is $1$ if $n\neq p^kd$ and $p^k$ if $n=p^kd$. Assume therefore that $p{\not|\:}d$. Applied to
$i=1$ this gives $d|n$ and we write $n=p^km$ with $p{\not|\:}m$. Assume first that $d \neq m$ and consider 
$$\left(\kern-1ex\begin{array}{c}n\\p^kd\end{array}\kern-1ex\right)_q 
= \frac{(q^n-1)(q^{n-1}-1)\cdots(q^{n-p^kd+1})}{(q-1)(q^2-1)\cdots(q^{p^kd}-1)}$$
Now, the multiplicity with which $\Phi_d(q)$ divides $q^{p^kd-j}-1$, for $0 \leq j < p^kd$ is equal to the same 
multiplicity for $q^{p^km-j}-1$ and hence the $\Phi_d(q)$-factors in the numerator and denominator cancel 
exactly. If instead $n=p^kd$, then we use $\binom{n}{p^{k-1}d}_q$ and the argument is the same except that we 
get an extra contribution of multiplicity $p^k$ in $q^n - 1$.
\end{proof}

Note that for a prime $p$ we have, by \cite[p. 160]{nag}
\begin{equation*}
\Phi_{np}(q) = 
\begin{cases}
  \Phi_n(q^p) & \quad \text{if $p {\:|\:} n$,}\\ \\
  \dfrac{\Phi_n(q^p)}{\Phi_n(q)} & \quad \text{if $p {\not|\:} n$.}
\end{cases}
\end{equation*}
So in characteristic $p$ we get $\Phi_{np^k}(q) = \Phi_n(q^{p^k})/\Phi_n(q^{p^{k-1}}) = \Phi_{n}(q)^{(p-1)p^{k-1}}$ if $p {\not|\:} n$.

\end{document}